\documentclass[11pt]{amsart}
\usepackage[dvipdfm]{hyperref}

\newtheorem{theorem}{Theorem}[section]
\newtheorem{lemma}[theorem]{Lemma}

\theoremstyle{definition}

\newtheorem{proposition}[theorem]{Proposition}
\newtheorem{corollary}[theorem]{Corollary}
\newtheorem{remark}[theorem]{Remark}

\theoremstyle{remark}

\newcommand{\bs}{\begin{split}}
\newcommand{\es}{\begin{split}}

\newcommand{\be}{\begin{equation}}
\newcommand{\ee}{\end{equation}}

\numberwithin{equation}{section}

\begin{document}

\title[An integral inequality on K\"{a}hler manifolds]
{An integral inequality for constant scalar curvature metrics on
K\"{a}hler manifolds}

\author{Ping Li}
\address{Department of Mathematics, Tongji University, Shanghai 200092, China}
\email{pingli@tongji.edu.cn}
\thanks{}




\begin{abstract}
We present in this note a lower bound for the Calabi functional in a
given K\"{a}hler class. This yields an integral inequality for
constant scalar curvature metrics, which can be viewed as a refined
version of Yau's Chern number inequality.
\end{abstract}

\maketitle
\section{Introduction and main results}
Suppose $(M,J)$ is an $n$-dimensional compact complex manifold with
a K\"{a}hler metric $g$. This $g$ determines a positive
$(1,1)$-form, the K\"{a}hler form
$\omega(\cdot,\cdot):=\frac{1}{2\pi}g(J\cdot,\cdot)$, and vice
versa. Under local coordinates $(z^1,\ldots,z^n)$, we have
$$g=(g_{i\bar{j}}):=
\big(g(\frac{\partial}{\partial z^i},\frac{\partial}{\partial
\bar{z^j}})\big),\qquad
\omega=\frac{\sqrt{-1}}{2\pi}g_{i\bar{j}}\text{d}z^i\wedge\text{d}\bar{z^j}.$$

Here and throughout the paper we use the Einstein convention for
various summations. Since K\"{a}hler metrics $g$ are one-to-one
correspondence to their K\"{a}hler forms $\omega$, hereafter we
don't distinguish them. Let $[\omega]\in H^{1,1}(M,\mathbb{R})$ be a
K\"{a}hler class and denote the set of all K\"{a}hler forms in
$[\omega]$ by $[\omega]^{+}$. In \cite{Ca}, Calabi introduced the
following functional, now called the Calabi functional:
$$[\omega]^{+}\rightarrow\mathbb{R},\qquad \omega\mapsto
\text{Ca}(\omega):=\int_Ms^2(\omega)\omega^n,$$

where $s(\omega):=2g^{i\bar{j}}r_{i\bar{j}}$ is the (Riemannian)
scalar curvature of the metric $\omega$. Here
$(g^{j\bar{i}}):=(g_{i\bar{j}})^{-1}$ and $r_{i\bar{j}}$ are the
component functions of
 $$\text{Ric}(\omega)=\frac{\sqrt{-1}}{2\pi}
r_{i\bar{j}}\text{d}z^i\wedge\text{d}\bar{z^j}:=
-\frac{\sqrt{-1}}{2\pi}\partial\bar{\partial}\text{log}\big(
\text{det}(g_{i\bar{j}})\big),$$

the Ricci form of $\omega$ which represents the first Chern class of
$M$.

Calabi proposed finding critical points of this functional as
candidates for canonical metrics, which are called extremal
K\"{a}hler metrics and can be viewed as a generalization of the
notion of constant scalar curvature K\"{a}hler (``cscK" for short)
metrics. It turns out that
$$\text{inf}_{\omega\in[\omega]^{+}}\text{Ca}(\omega),$$

 the greatest
lower bound of $\text{Ca}(\omega)$ in the K\"{a}hler class
$[\omega]$, is deeply related to the difficult open problem of
relating existence of cscK metrics in $[\omega]^{+}$ to
algebro-geometric stability (\cite{Do}, \cite{Ch}). However, aside
from this deep relationship, there is a rather trivial lower bound
on this functional involving only the cohomology class $[\omega]$
and the first Chern class $c_1$ due to the well-known fact that
$s(\omega)\omega^n=2nc_1(\omega)\wedge\omega^{n-1}$:
\be\label{inequalitycsc}\int_Ms^2(\omega)\omega^n\geq
\frac{\big(\int_Ms(\omega)\omega^n\big)^2}{\int_M\omega^n}=\frac{(2n
c_1[\omega]^{n-1})^2}{[\omega]^n},\ee

 where the equality holds if and only if $\omega$ is a cscK metric.

Recall that the K\"{a}hler curvature tensor $R$ of a K\"{a}hler
metric $\omega$ has an orthogonal decomposition under the $L^2$
norm, $R=S+P+B$, where $S$ is the scalar part, $P$ is the traceless
Ricci part, and $B$ is the Bochner curvature tensor. We call
$\omega$ a \emph{Bochner-K\"{a}hler} (``B-K" for short) metric if
$B\equiv0$. The metric $\omega$ is Einstein if and only if
$P\equiv0$. $(M,\omega)$ is a complex space form, i.e., $\omega$ has
constant holomorphic sectional curvature, if and only if
$P=B\equiv0$.

It is well-known via the Chern-Weil theory that the two integrals
\be\label{ingegral}\int_Mc_1^2(\omega)\wedge\omega^{n-2}\qquad
\text{and} \qquad\int_Mc_2(\omega)\wedge\omega^{n-2},\ee

where $c_1(\omega)$ \big($=\text{Ric}(\omega)$\big) and
$c_2(\omega)$ are the first two Chern forms of $\omega$, can be
expressed in terms of the $L^2$ norms of the above-mentioned
tensors. Apte should be the first one who derived the expression for
$c_2(\omega)\wedge\omega^{n-2}$ in \cite{Ap}. These two expressions
have many important related applications. For instance, together
with the Aubin-Yau theorem on K\"{a}hler-Einstein (``K-E" for short)
metrics with negative scalar curvature and the Calabi-Yau theorem on
Ricci-flat K\"{a}hler metrics, they can be led to Yau's remarkable
Chern number inequality (\cite{yau}). In \cite[p. 80, (2.80),
(2.80a)]{Be} of Besse's highly influential book, Apte's formula was
refined in terms of three terms: the scalar curvature $s(\omega)$
and the squared norms of the traceless tensor
$\tilde{\text{R}}\text{ic}(\omega):=\text{Ric}(\omega)-\frac{s(\omega)}{2n}\omega$,
and the Bochner curvature tensor $B$. In \cite[p. 80]{Be} these are
denoted by the symbols $s$, $\rho_0$ and $B_0$ respectively. Then
they use this expression and that of
$c_1^2(\omega)\wedge\omega^{n-2}$ to deduce (\cite[p. 80,
(2.82a)]{Be}) the expression of
$$\int_M\big(2(n+1)c_2(\omega)-nc_1^2(\omega)\big)\wedge\omega^{n-2},$$
which was applied in turn to deduce Yau's Chern number inequality in
\cite[p. 325]{Be}.

Here we would like to point out that the coefficient before the
traceless part $|\rho_0|^2$ in \cite[(2.80), (2.80a)]{Be} is
\emph{incorrect}. According to our detailed calculations (see
Section \ref{section3.1}), it should be $-\frac{2m}{m+2}$ rather
than $-\frac{2(m-1)}{m}$ (they denote by $m$ the complex dimension
of the K\"{a}hler manifold). Accordingly, the coefficient before
$|\rho_0|^2$ in \cite[(2.82a)]{Be} should be $-(1-\frac{2}{m+2})$
rather than $-(1-\frac{2}{m^2})$. However, this inaccuracy of the
coefficient before $|\rho_0|^2$ \emph{does not affect} the
above-mentioned application as the condition of $\omega$ being
K\"{a}hler-Einstein requires that $|\rho_0|=0$. As a byproduct of
correcting this coefficient, which will be done in details in
Section \ref{section3.1}, we will have a lower bound on the Calabi
functional as follows.

\begin{proposition}\label{proplowerbound}
Suppose $(M,\omega)$ is an $n$-dimensional compact K\"{a}hler
manifold with a K\"{a}hler metric $\omega$. Then we have
\be\label{inequalitybochner}
\int_Ms^2(\omega)\omega^n\geq8(n+1)\big(nc_1^2[\omega]^{n-2}-(n+2)c_2[\omega]^{n-2}\big),\ee

and the equality holds if and only if $\omega$ is a B-K metric.
\end{proposition}

\begin{remark}\label{remark1}~
\begin{enumerate}
\item
As we have commented above, the inaccuracy of the coefficient before
$|\rho_0|^2$ in \cite[(2.80a)]{Be} does not affect its deduction to
Yau's Chern number inequality. But the correctness of this
coefficient is crucial to our Proposition \ref{proplowerbound} and
the subsequent Theorem \ref{theorem}.

\item
A direct corollary of Proposition \ref{proplowerbound} is that a B-K
metric must be an extremal metric, which was proved by Matsumoto in
\cite{Ma}.

\item
Clearly the lower bound in (\ref{inequalitybochner}) makes no sense
unless
$$nc_1^2[\omega]^{n-2}>(n+2)c_2[\omega]^{n-2}.$$
Even if this holds, one may ask that whether in some cases this
lower bound is really sharper than the trivial one in
(\ref{inequalitycsc}). In Section \ref{section2}, we will use an
example of a Fano manifold, which was first noticed by Batyrev
(\cite{Ba}) to disprove an old conjecture, to illustrate that for
some K\"{a}hler classes of this manifold the lower bound in
(\ref{inequalitybochner}) is strictly larger than that in
(\ref{inequalitycsc}) and hence these two lower bounds are actually
independent to each other.
\end{enumerate}
\end{remark}

Although the appearance of the Bochner tensor $B$ in the
decomposition of the K\"{a}hler curvature tensor $R$ as an
irreducible summand under the unitary group has been known since the
1949 work of Bochner (\cite{Bo}), B-K metrics did not receive enough
attention for a long time until the work of Kamishima and Bryant
(\cite{Ka}, \cite{Br}). In particular, their  uniformization theorem
for \emph{compact} B-K manifolds tells us that the only
$n$-dimensional \emph{compact} B-K manifolds are the compact
quotients of the symmetric B-K manifolds $M^p_c\times M^{n-p}_{-c}$
(cf. \cite[p. 682]{Bo}), where $M^p_c$ denotes the $p$-dimensional
complex space form of constant holomorphic sectional curvature $c$.
An important corollary of this remarkable result is that any B-K
metric on a \emph{compact} complex manifold must be a cscK metric.
(This conclusion is \emph{not} valid for \emph{non-compact}
manifolds as Professor Bryant pointed out to me that there are B-K
metrics on $\mathbb{C}^n$ that are not cscK).

Now combining this result with (\ref{inequalitycsc}) and Proposition
\ref{proplowerbound} leads to the following integral inequality,
which provides an obstruction to the existence of cscK metrics in a
given K\"{a}hler class and is indeed a refinement of Yau's Chern
number inequality (see Corollary \ref{yauchernnumber} and its
proof).

\begin{theorem}\label{theorem}
Suppose the K\"{a}hler class $[\omega]$ of an $n$-dimensional
compact K\"{a}hler manifold contains a cscK metric. Then we have
\be\label{inequalityobstruction} n^2(c_1[\omega]^{n-1})^2\geq
2(n+1)\cdot[\omega]^n\cdot\big(nc_1^2[\omega]^{n-2}-(n+2)c_2[\omega]^{n-2}\big),\ee

where the equality holds if and only if $[\omega]$ contains a B-K
metric.
\end{theorem}

\begin{remark}\label{remark2}~
\begin{enumerate}
\item
Note that in the above conclusion the B-K metric in $[\omega]$ needs
not necessarily to be the original cscK metric in it. It is this
place that we need the fact that B-K metrics are cscK metrics in the
compact case.

\item
As we have mentioned in Item $(3)$ of Remark \ref{remark1}, for some
K\"{a}hler classes of the Fano manifold which will be described in
details in Section \ref{section2}, (\ref{inequalityobstruction})
does not hold and so these K\"{a}hler classes cannot contain cscK
metrics. Indeed, it can be shown that the holomorphic automorphism
group of this Fano manifold is not reductive and thus our conclusion
is consistent with the Matsushima-Lichnerowicz theorem (\cite{ML}).

\item
In \cite[\S 2.3]{Ti} of his famous lecture notes, Tian discussed
Yau's Chern number inequality along the line of the uniformization
theorem for constant holomorphic sectional curvature K\"{a}hler
manifolds. Indeed he has realized that the curvature integral
expressions for (\ref{ingegral}) can be used to give an integral
inequality for cscK metrics and provided one without a proof in
\cite[p. 21, Remark 2.15]{Ti}. In Section \ref{section3.2} we will
explain that how this inequality can be derived.
\end{enumerate}
\end{remark}

\begin{corollary}[Yau's Chern number inequality]\label{yauchernnumber}
Suppose $M$ is an $n$-dimensional compact K\"{a}hler manifold.

\begin{enumerate}
\item
If $c_1<0$, then
$$2(n+1)c_2(-c_1)^{n-2}\geq n(-c_1)^n,$$

 where the
equality holds if and only if $M$ is covered by the unit ball in
$\mathbb{C}^n$.

\item
If $c_1=0$, then $c_2[\omega]^{n-2}\geq0$ for any K\"{a}hler class
$[\omega]$, where the equality holds if and only if $M$ is covered
by a complex torus.
\end{enumerate}
\end{corollary}

\begin{proof}
If $c_1<0$, then $-c_1$ is a K\"{a}hler class and contains a K-E
metric by Aubin and Yau's theorem. Replacing the K\"{a}hler class
$[\omega]$ in (\ref{inequalityobstruction}) with $-c_1$ we obtain
$$2(n+1)c_2(-c_1)^{n-2}\geq n(-c_1)^n,$$

where the equality holds if and only if the K\"{a}hler class $-c_1$
contains a B-E metric, say $\omega$. By the above-mentioned
Kamishima-Bryant's result this K\"{a}hler metric $\omega$ is a cscK
metric and thus a K-E metric (\cite[p. 19, Prop. 2.12]{Ti}). This
means the traceless part tensor $P=0$ and the Bochner tensor $B=0$.
So this metric $\omega$ has negative constant holomorphic sectional
curvature and thus is covered by the unit ball in $\mathbb{C}^n$
(\cite[Theorem 1.12]{Ti}).

If $c_1=0$, then any K\"{a}hler class $[\omega]$ has a Ricci-flat
K\"{a}hler metric by the Calabi-Yau theorem. So
(\ref{inequalityobstruction}) tells us that
$c_2[\omega]^{n-2}\geq0$. The equality holds if and only if
$[\omega]$ contains a B-K metric, say $\omega$. Similar to the above
argument we know that this $\omega$ is a K-E metric and thus is
Ricci-flat. This means M is a compact K\"{a}hler manifold with
vanishing holomorphic sectional curvature and so it is covered by a
complex torus.
\end{proof}

\section{An example}\label{section2}
Let $\mathbb{P}^{n-1}$ denote the $(n-1)$-dimensional complex
projective space and
$$X:=\mathbb{P}\big(\mathcal{O}_{\mathbb{P}^{n-1}}
\oplus\mathcal{O}_{\mathbb{P}^{n-1}}(n-1)\big),$$

which is an $n$-dimensional Fano manifold. There was an old
conjecture asserting that, among all the $n$-dimensional Fano
manifolds $M$, the maximum of the top intersection number
$c_1^n(M)$, also called the degree of $M$, can only be attained by
$\mathbb{P}^n$. Namely, $c_1^n(M)\leq(n+1)^n$ and with equality if
and only if $M\cong\mathbb{P}^n$. In \cite{Ba}, Batyrev noticed that
$$c_1^n(X)=\frac{(2n-1)^n-1}{n-1}\sim\frac{2^ne^{-3/2}}{n}(n+1)^n$$

and thus disproved this conjecture. In \cite[p. 137-139]{De} Debarre
extended the construction of $X$ to a family of Fano manifolds and
used them to illustrate that there is indeed \emph{no} universal
polynomial upper bound on $\sqrt[n]{c_1^n(M)}$ among all the
$n$-dimensional Fano manifolds $M$.

In this section we will see that this $n$-dimensional Fano manifold
$X$ is also an ideal example for our purpose. More precisely, for
some K\"{a}hler classes of $X$ we can show that they don't satisfy
(\ref{inequalityobstruction}) and thus cannot contain cscK metrics.
Consequently, for these K\"{a}hler classes of $X$ the lower bound in
(\ref{inequalitybochner}) is sharper than that in
(\ref{inequalitycsc}) and hence clarify the non-triviality of
Proposition \ref{proplowerbound}.

Let $L$ and $H$ denote the first Chern classes of the line bundle
$\mathcal{O}_X(1)$ and the pull back of the hyperplane line bundle
$\mathcal{O}_{\mathbb{P}^{n-1}}(1)$ respectively. Then the
intersection ring of $X$ is generated by $L$ and $H$ with the
relations (cf. \cite[p. 138]{De})
\be\label{relation}L^2=(n-1)LH,\qquad LH^{n-1}=1,\qquad H^n=0.\ee

Standard calculation tells us that the first two Chern classes of
$X$ are as follows.
$$c_1=2L+H,\qquad c_2=2nLH-\frac{n(n-1)}{2}H^2.$$

With these data in hand, we can easily get the following lemma.

\begin{lemma}\label{lemmacompute}
Let $\Omega_{\alpha,\beta}:=\alpha L+\beta H$ ($\alpha,\beta>0$) be
a K\"{a}hler class of $X$. If we set
$$t:=(n-1)\alpha+\beta,$$
then we have
$$\Omega_{\alpha,\beta}^n=\frac{1}{n-1}(t^n-\beta^n),\qquad
c_1\Omega_{\alpha,\beta}^{n-1}=\frac{1}{n-1}[(2n-1)t^{n-1}-\beta^{n-1}],$$
$$c_1^2\Omega_{\alpha,\beta}^{n-2}=\frac{1}{n-1}[(2n-1)^2t^{n-2}-\beta^{n-2}]
,\qquad
c_2\Omega_{\alpha,\beta}^{n-2}=\frac{n}{2}[3t^{n-2}+\beta^{n-2}].$$
\end{lemma}

\begin{proof}
We only treat $\Omega_{\alpha,\beta}^n$ and
$c_1^2\Omega_{\alpha,\beta}^{n-2}$ and the other two cases are
similar. First note that the relation (\ref{relation}) implies
\begin{eqnarray}\label{relation2}
L^iH^{n-i}= \left\{ \begin{array}{ll} 0, & \text{if
$i=0$},\\
(n-1)^{i-1}, & \text{if $1\leq i\leq n$}.
\end{array} \right.
\end{eqnarray}

Therefore,
\be\begin{split} \Omega_{\alpha,\beta}^n&=(\alpha L+\beta H)^n\\
&=\sum_{i=0}^n\binom{n}{i}\alpha^i\beta^{n-i}L^iH^{n-i}\\
&=\sum_{i=1}^n\binom{n}{i}\alpha^i\beta^{n-i}(n-1)^{i-1}\qquad\big(\text{by
(\ref{relation2})}\big)\\
&=\frac{1}{n-1}\big\{[(n-1)\alpha+\beta]^n-\beta^n\big\}\\
&=\frac{1}{n-1}(t^n-\beta^n).
\end{split}\nonumber\ee

\be\begin{split} c_1^2\Omega_{\alpha,\beta}^{n-2}&=
(2L+H)^2(\alpha L+\beta H)^{n-2}\\
&=(4nLH+H^2)\sum_{i=0}^{n-2}\binom{n-2}{i}
\alpha^i\beta^{n-2-i}L^iH^{n-2-i}\\
&=4n\sum_{i=0}^{n-2}\binom{n-2}{i}
\alpha^i\beta^{n-2-i}L^{i+1}H^{n-1-i}+
\sum_{i=0}^{n-2}\binom{n-2}{i}
\alpha^i\beta^{n-2-i}L^iH^{n-i}\\
&=4n\sum_{i=0}^{n-2}\binom{n-2}{i}
\alpha^i\beta^{n-2-i}(n-1)^i+\sum_{i=1}^{n-2}\binom{n-2}{i}
\alpha^i\beta^{n-2-i}(n-1)^{i-1}\qquad\big(\text{by
(\ref{relation2})}\big)\\
&=4nt^{n-2}+\frac{1}{n-1}(t^{n-2}-\beta^{n-2})\\
&=\frac{1}{n-1}[(2n-1)^2t^{n-2}-\beta^{n-2}].
\end{split}\nonumber\ee
\end{proof}

This lemma leads to the following

\begin{proposition}\label{prop2}
Set
$$f(n,\alpha,\beta):=2(n+1)\Omega_{\alpha,\beta}^n[n
c_1^2\Omega_{\alpha,\beta}^{n-2}-(n+2)c_2
\Omega_{\alpha,\beta}^{n-2}]-n^2(c_1
\Omega_{\alpha,\beta}^{n-1})^2.$$

For arbitrary positive numbers $\alpha,\beta$, there exists a
positive integer $N(\alpha,\beta)$ such that
$$f(n,\alpha,\beta)>0
\qquad\text{whenever} \qquad n\geq N(\alpha,\beta).$$

Consequently, when $n\geq N(\alpha,\beta)$, the lower bound
(\ref{inequalitybochner}) for the Calabi functional in the
K\"{a}hler class $\Omega_{\alpha,\beta}$ is sharper than that in
(\ref{inequalitycsc}). This means the non-triviality of Proposition
\ref{proplowerbound} with respect to the trivial one in
(\ref{inequalitycsc}). Moreover, these K\"{a}hler classes
$\Omega_{\alpha,\beta}$ don't contain cscK metrics.
\end{proposition}

\begin{proof}
Using Lemma \ref{lemmacompute} we have \be\begin{split}
f(n,\alpha,\beta)=&\frac{2(n+1)}{n-1}(t^n-\beta^n)
\big[\frac{n(2n-1)^2}{n-1}t^{n-2}-\frac{n}{n-1}\beta^{n-2}
-\frac{3n(n+2)}{2}t^{n-2}-\frac{n(n+2)}{2}\beta^{n-2}\big]\\
&- \frac{n^2}{(n-1)^2}[(2n-1)t^{n-1}-\beta^{n-1}]^2.
\end{split}\nonumber\ee

Direct computation shows that
 \be\begin{split}
&\frac{(n-1)^2}{n}f(n,\alpha,\beta)\\
=& (n^3-2n^2-4n+8)t^{2n-2}+(n^3+2n^2)\beta^{2n-2}+(4n^2-2n)t^{n-1}
\beta^{n-1}\\
&-n(n+1)^2t^n\beta^{n-2}-(n+1)(5n^2-11n+8)t^{n-2}\beta^n.
\end{split}\nonumber\ee

Thus \be\begin{split} &\frac{(n-1)^2}{nt^{n-1}\beta^{n-1}}
f(n,\alpha,\beta)\\
=&(n^3-2n^2-4n+8)(\frac{t}{\beta})^{n-1}+
(n^3+2n^2)
(\frac{\beta}{t})^{n-1}+(4n^2-2n)\\
&-(n^2+n)(\frac{t}{\beta}) -(n+1)(5n^2-11n+8)(\frac{\beta}{t}).
\end{split}\ee

The fact that $\frac{t}{\beta}=(n-1)\frac{\alpha}{\beta}+1$ leads to
the desired result.
\end{proof}

\begin{remark}
It can be shown that the holomorphic automorphism group of $X$ is
not reductive (cf. \cite[p. 138]{De}). So the nonexistence of cscK
metrics in these K\"{a}hler classes can also be obtained via the
Matsushima-Lichnerowicz theorem (\cite{ML}). At the time of writing
this note the author is not able to find out a compact K\"{a}hler
manifold $(M,\omega)$ whose holomorphic automorphism group is
reductive such that the K\"{a}hler class $[\omega]$ does not satisfy
(\ref{inequalityobstruction}).
\end{remark}

\section{Proof of Proposition \ref{proplowerbound} and related
remarks}
\subsection{Proof of Proposition
\ref{proplowerbound}}\label{section3.1}

Suppose $M$ is a compact $n$-dimensional K\"{a}hler manifold
 with a K\"{a}hler metric $g$. Under local complex coordinates
 $(z^1,\ldots,z^n)$, we write the K\"{a}hler metric $g$,
 its K\"{a}hler form $\omega$, the $(4,0)$-type K\"{a}hler curvature tensor
 $R$, the Ricci form $\text{Ric}(\omega)$, and the (Riemannian) scalar curvature
 $s(\omega)$ as follows.
$$g=(g_{i\bar{j}}):=
\big(g(\frac{\partial}{\partial z^i},\frac{\partial}{\partial
\bar{z^j}})\big),\qquad\omega=\frac{\sqrt{-1}}{2\pi}g_{i\bar{j}}\text{d}z_i\wedge\text{d}\bar{z_j},
\qquad R_{i\bar{j}k\bar{l}}:=R(\frac{\partial}{\partial
z_i},\frac{\partial}{\partial \bar{z_j}}, \frac{\partial}{\partial
z_k},\frac{\partial}{\partial \bar{z_l}})$$
$$\text{Ric}(\omega)=\frac{\sqrt{-1}}{2\pi}
r_{i\bar{j}}\text{d}z_i\wedge\text{d}\bar{z_j}=\frac{\sqrt{-1}}{2\pi}
g^{kl}R_{i\bar{j}k\bar{l}}\text{d}z_i\wedge\text{d}\bar{z_j},
$$
$$s(\omega):=2g^{i\bar{j}}r_{i\bar{j}},\qquad
(g^{j\bar{i}}):=(g_{i\bar{j}})^{-1}.$$

The pointwise squared norms of $R$ and $\text{Ric}(\omega)$ are
defined as follows:
$$|R|^2:=R_{i\bar{j}k\bar{l}}R_{p\bar{q}r\bar{s}}
g^{i\bar{q}}g^{p\bar{j}}g^{k\bar{s}}g^{r\bar{l}},
\qquad|\text{Ric}(\omega)|^2:=r_{i\bar{j}}r_{p\bar{q}}g^{i\bar{q}}g^{p\bar{j}}.$$

It is well-known that these norms are independent of the choices of
local coordinates. With these notions and symbols understood, we
have the following two well-known facts, which essentially should be
due to Apte (\cite{Ap}). A detailed proof can be found in \cite[p.
225-226]{Zh}.

\begin{lemma}[Apte]\label{lemmaapte}
\be\label{formulaapte1}\text{Ric}^2(\omega)\wedge\omega^{n-2}=
\big[\frac{s^2(\omega)}{4}-|\text{Ric}(\omega)|^2\big]\cdot\frac{\omega^n}{n(n-1)}\ee
\be\label{formulaapte2}c_2(\omega)\wedge\omega^{n-2}=
\big[\frac{s^2(\omega)}{4}-2|\text{Ric}(\omega)|^2+|R|^2\big]\cdot\frac{\omega^n}{2n(n-1)}\ee
\end{lemma}

\begin{remark}~
\begin{enumerate}
\item
In \cite{Zh}, the K\"{a}hler form is defined to be
$\frac{\sqrt{-1}}{2}g_{i\bar{j}}\text{d}z_i\wedge\text{d}\bar{z_j}$
and so there is an extra factor $\pi^2$ in the expressions.

\item
Apte only stated (\ref{formulaapte2}) in \cite[p. 150]{Ap} in a
slightly different form. In her notation in \cite{Ap}, ``$R$"
denotes the scalar curvature $s(\omega)$ and ``$R_{ij}R^{ij}$"
denotes the pointwise squared norm of the Ricci tensor. Namely, if
we denote by $g_{ij}$ and $r_{ij}$ the Riemannian metric $g$ and the
Ricci tensor under local \emph{real} coordinates, ``$R_{ij}R^{ij}$"
means $r_{ij}r_{pq}g^{ip}g^{jq}$, which can be shown to be exactly
twice of $|\text{Ric}(\omega)|^2$.
\end{enumerate}
\end{remark}

 Under the local coordinates $(z^1,\ldots,z^n)$, the decomposition $R=S+P+B$ has the following expression (cf. \cite[p. 86]{Bo}).
$$R_{i\bar{j}k\bar{l}}=S_{i\bar{j}k\bar{l}}+P_{i\bar{j}k\bar{l}}+B_{i\bar{j}k\bar{l}},$$
where
$$S_{i\bar{j}k\bar{l}}=\frac{s(\omega)}{2n(n+1)}(g_{i\bar{j}}g_{k\bar{l}}+g_{i\bar{l}}g_{k\bar{j}}),
$$
$$P_{i\bar{j}k\bar{l}}=\frac{1}{n+2}
(g_{i\bar{j}}\tilde{r}_{k\bar{l}}+
g_{k\bar{l}}\tilde{r}_{i\bar{j}}+g_{i\bar{l}}
\tilde{r}_{k\bar{j}}+g_{k\bar{j}}\tilde{r}_{i\bar{l}}),\qquad
\tilde{r}_{i\bar{j}}:=r_{i\bar{j}}-\frac{s(\omega)}{2n}g_{i\bar{j}}.$$

Clearly $\omega$ is a Einstein metric if and only if the traceless
$(1,1)$-form
$$\tilde{\text{R}}\text{ic}(\omega):=
\text{Ric}(\omega)-\frac{s(\omega)}{2n}\omega=\frac{\sqrt{-1}}{2\pi}
\tilde{r}_{i\bar{j}}\text{d}z_i\wedge\text{d}\bar{z_j}\equiv0.$$

\begin{lemma}\label{propnorm}
The pointwise squared norms of $\textrm{Ric}(\omega)$,
$\tilde{R}ic(\omega)$, $S$ and $P$ satisfy
\be\label{eq0}|\text{Ric}(\omega)|^2=|\tilde{\text{R}}\text{ic}(\omega)|^2+\frac{s^2(\omega)}{4n},\qquad
|S|^2=\frac{s^2(\omega)}{2n(n+1)},\qquad
|P|^2=\frac{4}{n+2}|\tilde{\text{R}}\text{ic}(\omega)|^2,\ee

where $|\tilde{\text{R}}\text{ic}(\omega)|^2:=\tilde{r}_{i\bar{j}}
\tilde{r}_{p\bar{q}}g^{i\bar{q}}g^{p\bar{j}}$ and the definitions of
$|S|^2$ and $|P|^2$ are similar to that of $|R|^2$.
\end{lemma}

\begin{proof}
The first two are well-known to experts. Since the coefficient
$4/(n+2)$ in the third one is crucial to our later use, here we give
a detailed computation for it. For simplicity we can choose a local
unitary frame field and so assume that $g_{i\bar{j}}=\delta_{ij}$.
Thus

\be\label{eq}\begin{split}|P|^2=&P_{i\bar{j}k\bar{l}}P_{j\bar{i}l\bar{k}}\\
=&\frac{1}{n+2}(\delta_{ij}\tilde{r}_{k\bar{l}}+
\delta_{kl}\tilde{r}_{i\bar{j}}+\delta_{il}
\tilde{r}_{k\bar{j}}+\delta_{kj}\tilde{r}_{i\bar{l}})P_{j\bar{i}l\bar{k}}\\
=&\frac{1}{n+2}(\delta_{ij}\tilde{r}_{k\bar{l}}P_{j\bar{i}l\bar{k}}+\text{three
other terms}).
\end{split}\ee

Note that under our assumption we have
$$\sum\tilde{r}_{i\bar{j}}\tilde{r}_{j\bar{i}}=|\tilde{\text{R}}\text{ic}(\omega)|^2\qquad
\text{and}\qquad \sum\tilde{r}_{i\bar{i}}=0.$$

Therefore,
 \be\begin{split}
\delta_{ij}\tilde{r}_{k\bar{l}}\cdot(n+2)P_{j\bar{i}l\bar{k}}&=
\delta_{ij}\delta_{ji}\tilde{r}_{k\bar{l}}\tilde{r}_{l\bar{k}}+
\delta_{ij}\delta_{lk}\tilde{r}_{k\bar{l}}\tilde{r}_{j\bar{i}}+
\delta_{ij}\delta_{jk}\tilde{r}_{k\bar{l}}\tilde{r}_{l\bar{i}}+
\delta_{ij}\delta_{li}\tilde{r}_{k\bar{l}}\tilde{r}_{j\bar{k}}\\
&=n|\tilde{\text{R}}\text{ic}(\omega)|^2+0+|\tilde{\text{R}}\text{ic}(\omega)|^2
+|\tilde{\text{R}}\text{ic}(\omega)|^2\\
&=(n+2)|\tilde{\text{R}}\text{ic}(\omega)|^2.\end{split}\nonumber\ee
The situations for the ``other three terms" in (\ref{eq}) are
similar. Thus we have
$$|P|^2=P_{i\bar{j}k\bar{l}}P_{j\bar{i}l\bar{k}}=
\frac{1}{n+2}\cdot4\cdot|\tilde{\text{R}}\text{ic}(\omega)|^2=
\frac{4}{n+2}|\tilde{\text{R}}\text{ic}(\omega)|^2.$$
\end{proof}

Now we can prove our Proposition \ref{proplowerbound} via the
following

\begin{proposition}
\be\label{integral1} c_1^2[\omega]^{n-2}=
\int_M\big[\frac{s^2(\omega)}{4n^2}-\frac{
|\tilde{\text{R}}\text{ic}(\omega)|^2}{n(n-1)}\big]\cdot\omega^n \ee
 \be\label{integral2}
 c_2[\omega]^{n-2}=\int_M\big[\frac{s^2(\omega)}{8n(n+1)}-\frac{
|\tilde{\text{R}}\text{ic}(\omega)|^2}{(n+2)(n-1)}+\frac{|B|^2}{2n(n-1)}\big]\cdot\omega^n
 \ee

 In particular,
 \be\label{integral3}
nc_1^2[\omega]^{n-2}-(n+2)c_2[\omega]^{n-2}
=\int_M\frac{s^2(\omega)}{8(n+1)}\cdot\omega^n-
\int_M\frac{n+2}{2n(n-1)}|B|^2\cdot\omega^n\ee

and thus Proposition \ref{proplowerbound} holds.
\end{proposition}

\begin{proof}
Integrating (\ref{formulaapte1}) and (\ref{formulaapte2}) over $M$
and using the relations (\ref{eq0}) to replace the terms
$|\text{Ric}(\omega)|^2$ and $|R|^2$ lead to (\ref{integral1}) and
(\ref{integral2}).
\end{proof}

Now we can correct the coefficient in \cite[(2.80a)]{Be} by
rewriting (\ref{integral1}) and (\ref{integral2}) as follows.
$$\frac{1}{(n-2)!}c_1^2[\omega]^{n-2}=
\int_M\big[\frac{n-1}{4n}s^2(\omega)-|\tilde{\text{R}}\text{ic}(\omega)|^2\big]\cdot\frac{\omega^n}{n!}$$
$$\frac{1}{(n-2)!}c_2[\omega]^{n-2}=\frac{1}{2}\int_M\big[\frac{n-1}{4(n+1)}s^2(\omega)
-\frac{2n}{n+2}|\tilde{\text{R}}\text{ic}(\omega)|^2+
|B|^2\big]\cdot\frac{\omega^n}{n!}$$

Note that in \cite[p. 80]{Be} the K\"{a}hler form is defined to be
$2\pi\omega$ in our notation and so there is an additional factor
$4\pi^2$. Clearly the correctness of the coefficient
$-\frac{2n}{n+2}$ before $|\tilde{\text{R}}\text{ic}(\omega)|^2$ is
crucial to establish (\ref{integral3}) and Proposition
\ref{proplowerbound}.

\subsection{On a remark of Tian}\label{section3.2}
In Chapter 2 of his lecture notes \cite{Ti}, Tian discussed Yau's
Chern number inequality along the line of the uniformization theorem
for K\"{a}hler manifolds with constant holomorphic sectional
curvature. At the end of chapter 2 (\cite[Remark 2.15]{Ti}), he
remarked that, if an $n$-dimensional compact K\"{a}hler manifold
$(M,\omega)$ has constant scalar curvature $s(\omega)$, then  we
have the following integral inequality
\be\label{Tian}c_1^2[\omega]^{n-2}-c_2[\omega]^{n-2}\leq\frac{n+2}{8n^2(n+1)}s^2(\omega)[\omega]^n,\ee

where the equality holds if and only if $\omega$ is of constant
holomorphic sectional curvature. Using our notation the K\"{a}hler
form and the scalar curvature in \cite{Ti} are defined to be
$\pi\omega$ and one half of $s(\omega)$ respectively.

Indeed, (\ref{Tian}) can be proved by using (\ref{integral1}) and
(\ref{integral2}) directly:
\be\begin{split}
c_1^2[\omega]^{n-2}-c_2[\omega]^{n-2}&=\int_M\big[\frac{n+2}
{8n^2(n+1)}s^2(\omega)-\frac{2 }{n(n+2)(n-1)}|\tilde{\text{R}}
\text{ic}(\omega)|^2-\frac{|B|^2}{2n(n-1)}\big]\omega^n\\
&\leq\int_M\frac{n+2}
{8n^2(n+1)}s^2(\omega)\omega^n\\
&=\frac{n+2}{8n^2(n+1)}s^2(\omega)[\omega]^n,\qquad(\text{$\omega$
is cscK})\end{split}\nonumber\ee

where the equality holds if and only if $P=B=0$ and so $\omega$ has
constant holomorphic sectional curvature.

Under the assumption that $\omega$ be a cscK metric, we can compare
(\ref{inequalityobstruction}) and (\ref{Tian}) by rewriting
(\ref{inequalityobstruction}) as follows. \be\label{li}\begin{split}
c_1^2[\omega]^{n-2}-c_2[\omega]^{n-2}&\leq\frac{n(c_1[\omega]^{n-1})^2}
{2(n+1)[\omega]^n}+\frac{2}{n}c_2[\omega]^{n-2}\\
&=\frac{1}{8n(n+1)}s^2(\omega)[\omega]^n+\frac{2}{n}c_2[\omega]^{n-2}.\end{split}\ee

Now the difference of the upper bounds in (\ref{li}) and
(\ref{Tian}) is \be\begin{split}
&\frac{2}{n}c_2[\omega]^{n-2}-\frac{2}{8n^2(n+1)}s^2(\omega)[\omega]^n\\
=&\frac{2}{n}\big\{c_2[\omega]^{n-2}-\frac{1}{8n(n+1)}s^2(\omega)[\omega]^n\big\}\\
=&\int_M\big[-\frac{2
|\tilde{\text{R}}\text{ic}(\omega)|^2}{n(n+2)(n-1)}+\frac{|B|^2}{n^2(n-1)}\big]\cdot\omega^n\qquad\big(\text{by
(\ref{integral2})}\big),
\end{split}\nonumber\ee

whose sign can be either negative or positive.
\bibliographystyle{amsplain}

\end{document}